\documentclass[letterpaper, 11 pt]{article}
\usepackage{fullpage,amsthm, amsmath, amssymb, amsfonts}
\usepackage{graphics, epstopdf}
\usepackage{pstricks}
\usepackage{psfrag}
\usepackage{graphicx}
\usepackage{auto-pst-pdf}

\newcommand{\RR}{\mathbb{R}}

\newcommand{\newword}[1]{\textbf{\emph{#1}}}

\newcommand{\C}{\mathcal{C}}
\newcommand{\W}{\mathcal{W}}
\newcommand{\M}{\mathcal{M}}

\newcommand{\I}{\mathcal{I}}
\newcommand{\B}{\mathcal{B}}

\def\M{\mathcal{M}}

\newtheorem{conj}{Conjecture}[section]

\newtheorem{Theorem}[conj]{Theorem}
\newtheorem{theorem}[conj]{Theorem}

\newtheorem{proposition}[conj]{Proposition}

\newtheorem{lemma}[conj]{Lemma}

\newtheorem{corollary}[conj]{Corollary}

\theoremstyle{definition}

\begin{document}

\title{Links in the complex of weakly separated collections}
\author{SuHo Oh and David E Speyer}
\date{}
\maketitle

\begin{abstract}

Plabic graphs are interesting combinatorial objects used to study the totally nonnegative Grassmannian. 
Faces of plabic graphs are labeled by $k$-element sets of positive integers, and a collection of such $k$-element sets are the face labels of a plabic graph if that collection forms a maximal weakly separated collection. 
There are moves that one can apply to plabic graphs, and thus to maximal weakly separated collections, analogous to mutations of seeds in cluster algebras. In this short note, we show that if two maximal weakly separated collections can be mutated from one to another, then one can do so while freezing the face labels they have in common.
\end{abstract}


%
%

\section{Introduction}
\label{sec:WS}

Fix two positive integers $k \leq n$.
Let $[n]:= \{ 1,\ldots, n \}$.
We will generally consider $[n]$ as cyclically ordered.  We will say that
$i_1$, $i_2$, \ldots, $i_r$ in $[n]$ are {\it cyclically ordered\/} 
if $i_s < i_{s+1} < \cdots < i_r < i_1 < i_2 < \cdots < i_{s-1}$ for some $s\in[r]$. 


Fix positive integers $k<n$. Let $I$ and $J$ be two $k$-element subsets of $\{ 1,2, \ldots, n \}$.
The following definition is due to Leclerc and Zelevinsky~\cite{LZ}, see also~\cite{Scott} and~\cite{OPS}:
The sets $I$ and $J$ are called \newword{weakly separated} if there do \textbf{not} exist $a$, $b$, $c$ and $d$ cyclically ordered with $a$, $c \in I \setminus J$ and $b$, $d \in J \setminus I$. 
Graphically, $I$ and $J$ are weakly separated if we can draw a chord across the circle separating $I \setminus J$ from $J \setminus I$.
We write $I \parallel J$ to indicate that $I$ and $J$ are weakly separated.

Write $\binom{[n]}{k}$ for the set of $k$ element subsets of $[n]$. We will use the term \newword{collection} to refer to a subset of $\binom{[n]}{k}$.
We define a \newword{weakly separated collection} to be a collection 
$\C\subset \binom{[n]}{k}$ such that, for any $I$ and $J$ in $\C$, 
the sets $I$ and $J$ are weakly separated. 
We define a \newword{maximal weakly separated collection} to be a weakly separated collection which is not contained 
in any other weakly separated collection.

Following Leclerc and Zelevinsky \cite{LZ}, Scott observed the following:
\begin{proposition}
\cite{Scott}, cf.\ \cite{LZ}
\label{mutation}
Let $S \in \binom{[n]}{k-2}$ and let $a$, $b$, $c$, $d$ be cyclically ordered elements of $[n] \setminus S$. 
Suppose that a maximal weakly separated collection $\C_1$ contains $S \cup \{ a, b \}$, $S \cup \{b, c \}$, $S \cup \{c, d \}$, $S \cup \{ d, a \}$ and $S \cup \{ a, c \}$.
Then $\C_2 := (\C_1 \setminus \{ S \cup \{ a, c \} \}) \cup \{ S \cup \{ b, d \} \}$ is also a maximal weakly separated collection.
\end{proposition}
When $\C_1$ and $\C_2$ are related as in this proposition, we will say that $\C_1$ and $\C_2$ are \newword{mutations} of each other. 
Relying on results of \cite{Post}, in~\cite{OPS} the authors proved that any two maximal weakly separated collections are linked by a sequence of mutations.
As a corollary, any two maximal weakly separated collections have the same cardinality -- namely $k(n-k)+1$.

In other words, if we form a simplicial complex whose vertices are indexed by $\binom{[n]}{k}$, and whose faces are the maximal weakly separated sets, then this complex is pure of dimension $k(n-k)$ and is connected in codimension $1$.
This complex was further studied in \cite{HH}.

In this paper, we will study the links of faces in this complex. Namely, our main result is:
\begin{theorem}
Let $\B \subset \binom{[n]}{k}$ be a weakly separated collection. Let $\C$ and $\C'$ be two maximal weakly separated collections containing $\B$. Then $\C$ and $\C'$ are linked by a chain of mutations $\C = \C_1 \to \C_2 \to \cdots \to \C_r = \C'$ where all the $\C_i$ contain $\B$.
\end{theorem}

In other words, if $\sigma$ is a face of the simplicial complex described above, with codimension greater than $1$, then the link of $\sigma$ is connected in codimension $1$.

Even the case $\B = \emptyset$, where this result is due to Postnikov~\cite{Post}, our proof is new and independent of Postnikov's.

\section{Notations}

We will use the following notations through out the paper:
We write $(a, b)$ for the open cyclic interval from $a$ to $b$. In other words,
the set of $i$ such that $a$, $i$, $b$ is cyclically ordered.  We write $[a,
b]$ for the closed cyclic interval, $[a, b] = (a, b) \cup \{ a, b \}$, and use
similar notations for half open intervals.

If $S$ is a subset of $[n]$ and $a$ an
element of $[n]$, we may abbreviate $S \cup \{ a \}$ and $S \setminus \{ a \}$
by $S a$ and $S \setminus a$. 

In this paper, we need to deal with three levels of objects: elements of $[n]$,
subsets of $[n]$, and collections of subsets of $[n]$. For clarity, we will
denote these by lower case letters, capital letters, and calligraphic letters,
respectively. 

The use of the notation $I \setminus J$ does not imply $J \subseteq I$.

\section{Positroids}

More generally,~\cite{OPS} studied weakly separated collections within positroids.
We review this material briefly now; see~\cite{Post} and~\cite{OPS} for more.
A \newword{Grassmann necklace} is a sequence $\I = (I_1, \cdots, I_n)$ of $k$-element subsets of $[n]$ such that, for $i \in [n]$, the set $I_{i+1}$ contains $I_i \setminus \{ i \}$.  (Here the indices are taken modulo $n$.)
If $i \not \in I_{i}$, then we should have $I_{i+1} = I_i$.

Define a linear order $<_i$ on $[n]$ by
$$i <_i i+1 <_i i+2 <_i \cdots <_i i-1.$$
We extend  $<_i$ to $k$ element sets, as follows.
For $I=\{i_1, \cdots, i_k \}$ and $J=\{j_1, \cdots, j_k \}$ with 
$i_1 <_i i_2 \cdots <_i i_k$ and $j_1 <_i j_2 \cdots <_i j_k$, define the partial order
$$
I \leq_i J 
\textrm{ if and only if } i_1 \leq_i j_1, \cdots, i_k \leq_i j_k.
$$
Given a Grassmann necklace $\I=(I_1,\cdots,I_n)$, define the \newword{positroid} $\M_{\I}$ to be
$$
\M_{\I} := \{J \in \binom{[n]}{k} \mid I_i \leq_i J \text{ for all } i \in [n] \}.
$$

Fix a Grassmann necklace $\I=(I_1,\cdots,I_n)$, with corresponding positroid
$\M_{\I}$. Then $\C$ is called a {\it weakly separated collection inside}
$\M_{\I}$ if $\C$ is a weakly separated collection and $\I \subseteq \C
\subseteq \M_{\I}$. We call $\C$ a {\it maximal weakly separated collection
inside} $\M_{\I}$ if it is maximal among weakly separated collections inside
$\M_{\I}$.  

Our actual main result is
\begin{theorem}
Let $\I$ be a Grassmann necklace and let $\B$ be a weakly separated collection in $\M_{\I}$. Let $\C$ and $\C'$ be two maximal weakly separated collections in $\M_{\I}$ containing $\B$. Then $\C$ and $\C'$ are linked by a chain of mutations $\C = \C_1 \to \C_2 \to \cdots \to \C_r = \C'$ where all the $\C_i$ contain $\B$ and are in $\M_{\I}$.
\end{theorem}

The case $I_i = \{ i,i+1, \ldots, i+k-1 \}$ corresponds to taking $\M_{\I}$ to be all of $\binom{[n]}{k}$. This result also implies the main result of \cite{DKK}.

\section{Plabic tilings} \label{sec:tilings}

In this section, we review the plabic tiling construction from~\cite{OPS}.
The motivation for this construction is as follows: The main result of~\cite{OPS} is that maximal weakly separated collections are in bijection with certain planar bipartite graphs called ``reduced plabic graphs". The planar dual of a reduced plabic graph is thus a bi-colored CW complex, homeomorphic to a two-dimensional disc. 
The plabic tiling construction assigns a  bi-colored two-dimensional CW complex to any weakly separated collection, maximal or not. 
For the purposes of this picture, we only need plabic tilings, not plabic graphs.

\begin{figure}
\centerline{
\begin{postscript}
\scalebox{0.6}{
\psfrag{123}[cc][cc]{\Large $123$}
\psfrag{234}[cc][cc]{\Large $234$}
\psfrag{345}[cc][cc]{\Large $345$}
\psfrag{456}[cc][cc]{\Large $456$}
\psfrag{567}[cc][cc]{\Large $567$}
\psfrag{167}[cc][cc]{\Large $167$}
\psfrag{127}[cc][cc]{\Large $127$}
\psfrag{134}[cc][cc]{\Large $134$}
\psfrag{135}[cc][cc]{\Large $135$}
\psfrag{136}[cc][cc]{\Large $136$}
\psfrag{126}[cc][cc]{\Large $126$}
\psfrag{356}[cc][cc]{\Large $356$}
\psfrag{156}[cc][cc]{\Large $156$}
\psfrag{v1}[cc][cc]{\Large $v_1$}
\psfrag{v2}[cc][cc]{\Large $v_2$}
\psfrag{v3}[cc][cc]{\Large $v_3$}
\psfrag{v4}[cc][cc]{\Large $v_4$}
\psfrag{v5}[cc][cc]{\Large $v_5$}
\psfrag{v6}[cc][cc]{\Large $v_6$}
\psfrag{v7}[cc][cc]{\Large $v_7$}
\includegraphics{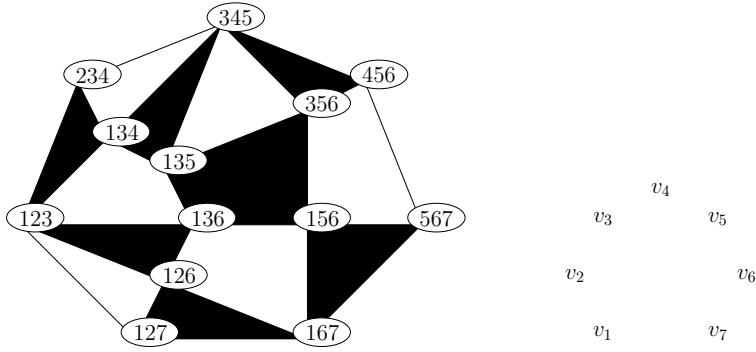}}
\end{postscript}
}
\caption{Example of a plabic tiling}\label{TilingExample}
\end{figure}

Let us fix $\C$, a weakly separated collection in $\M_{\I}$. For $I$ and $J \in \M_{\I}$, say that $I$ \newword{neighbors} $J$ if 
$$|I \setminus J|=|J \setminus I|=1.$$
Let $K$ be any $(k-1)$ element subset of $[n]$. We define the \newword{white clique} $\W(K)$ to be the set of $I \in \C$ such that $K \subset I$. Similarly, for $L$ a $(k+1)$ element subset of $[n]$, we define the \newword{black clique} $\B(L)$ for the set of $I \in \C$ which are contained in $L$. We call a clique \newword{nontrivial} if it has at least three elements. 
Observe that, if $\mathcal{X}$ is a nontrivial clique, then it cannot be both black and white. 


Observe that a white clique $\W(K)$ is of the form $\{ K a_1, K a_2, \ldots, K a_r \}$ for some $a_1$, $a_2$, \dots, $a_r$, which we take to be cyclically ordered. Similarly, $\B(L)$ is of the form $\{ L \setminus b_1, L \setminus b_2, \ldots, L \setminus b_s \}$, with the $b_i$'s cyclically ordered. 
If $\W(K)$ is nontrivial, we define the boundary of $\W(K)$ to be the cyclic graph 
$$(K a_1) \to (K a_2) \to \cdots \to (K a_r) \to (K a_1).$$ 
Similarly, the boundary of a nontrivial $\B(L)$ is 
$$(L \setminus b_1) \to (L \setminus b_2) \to \cdots \to (L \setminus b_s) \to (L \setminus b_1).$$ 
If $(J, J')$ is a two element clique, then we define its boundary to be the graph with a single edge $(J,J')$; we define an one element clique to have empty boundary.


We now define a two dimensional CW-complex $\Sigma(\C)$.
The vertices of $\Sigma(\C)$ will be the elements of $\C$.
There will be an edge $(I,J)$ if
\begin{enumerate}
\item $\W(I \cap J)$ is nontrivial and $(I,J)$ appears in the boundary of $\W(I \cap J)$ or
\item $\B(I \cup J)$ is nontrivial and $(I,J)$ appears in the boundary of $\B(I \cup J)$ or
\item $\W(I \cap J) = \B(I \cup J)  = \{ I, J \}$.
\end{enumerate}
There will be a (two-dimensional) face of $\Sigma(\C)$ for each nontrivial clique $\mathcal{X}$ of $\C$. 
The boundary of this face will be the boundary of $\mathcal{X}$. 
We will refer to each face of $\Sigma(\C)$ as \newword{black} or \newword{white}, according to the color of the corresponding clique. We call a CW-complex of the form $\Sigma(\C)$ a \newword{plabic tiling}. An implicit claim here is that, if $\W(I \cap J)$ and $\B(I \cup J)$ are both nontrivial, then $(I,J)$ is a boundary edge of both, so that $2$-dimensional faces of $\Sigma(\C)$ are glued along boundary edges. This is not obvious, but it is true; see~\cite[Lemma~9.2]{OPS}.

So far, $\Sigma(\C)$ is an abstract CW-complex. Our next goal is to embed it in a plane.

Fix $n$ points $v_1$, $v_2$, \dots, $v_n$ in $\RR^2$, at the vertices of a convex $n$-gon in clockwise order. 
Define a linear map $\pi: \RR^n \to \RR^2$ by $e_{a} \mapsto v_{a}$. 
For $I \in \binom{[n]}{t}$, set $e_I = \sum_{a \in I} e_{a}$. 
We abbreviate $\pi(e_I)$ by $\pi(I)$.

%

We extend the map $\pi$ to a map from $\Sigma(\C)$ to $\RR^2$ as follows: Each vertex $I$ of $\Sigma(\C)$ is sent to $\pi(I)$ and each face of $\Sigma(\C)$ is sent to the convex hull of the images of its vertices. 
We encourage the reader to consult Figure~\ref {TilingExample} and see that the vector $\pi(Si) - \pi(Sj)$ is a translation of $v_i - v_j$.\footnote{This figure is extremely similar to  \cite[Figure 9]{OPS}, we have redrawn it to avoid issues of figure reuse.}

 We define $\pi(\I)$ to be the closed polygonal curve whose vertices are, in order, $\pi(I_1)$, $\pi(I_2)$, \dots, $\pi(I_n)$, $\pi(I_1)$.
 We now summarize the main results of \cite{OPS} concerning plabic tilings:

\begin{proposition} [{\cite[Prop. 9.4, Prop. 9.8, Prop 9.10,  Theorem 11.1]{OPS}}] \label{topsame}
With the above notation, $\pi(\I)$ is a simple closed curve, except that if $\I$ has repeated elements then $\pi(\I)$ may touch itself at those vertices, in a manner which can be perturbed to a simple closed curve.
If $\C$ is a weakly separated collection in $\M_{\I}$, then the map $\pi: \Sigma(\C) \to \RR^2$ is injective, and its image lands inside the curve $\pi(\I)$

The collection $\C$ is maximal among weakly separated collections in $\M_{\I}$ if and only if $\Sigma(\C)$ fills the entire interior of the curve $\pi(\I)$. 

If $J$ is weakly separated from all elements of $\I$, then $J \in \M_{\I}$ if and only if $\pi(J)$ is inside the curve $\pi(\I)$.
\end{proposition}

We will sometimes speak of \newword{triangulating} $\Sigma(\C)$, meaning to take each $2$-cell of $\Sigma(\C)$ and divide it into triangles. 
Coloring these triangles with the color of the corresponding $2$-cells, the vertices of a white triangle are of the form $(Sa, Sb, Sc)$ for some $k-1$ element set $S$ and some $a$, $b$, $c \in [n] \setminus S$. The vertices of a black triangle are of the form $(S \setminus a, S \setminus b, S \setminus c)$ for some $k+1$ element set $S$ and some $a$, $b$, $c \in S$. 
Note that the image of a triangle under $\pi$ is a translate of $\mathrm{Hull}(v_a, v_b, v_c)$ or $\mathrm{Hull}(-v_a, -v_b, -v_c)$ respectively.
The triangle is oriented clockwise if $(a,b,c)$ are cyclically ordered.

\section{A lemma regarding mutations}

We will need the following lemma.
\begin{lemma} 
\label{lem:cross}
Let $H$ be a subset of $[n]$ of cardinality $k-2$; let $a$, $b$, $c$, $d$ be circularly ordered elements in $[n] \setminus H$.
Let $J$ be another $k$ element subset of $[n]$.
Suppose that $Hac$ and $Hbd$ are weakly separated with $J$. Then $Hab$, $Hbc$, $Hcd$ and $Hda$ are weakly separated with $J$.
\end{lemma}

The relevance of this lemma is as follows: Suppose that $\C$ is a weakly separated collection which contains $Hac$ and all of whose elements other than $Hac$ are weakly separated from $Hbd$. 
Then the lemma shows that $\C' := \C \cup \{ Hab, Hbc, Hcd, Had \}$ is weakly separated. Extending $\C'$ to some maximal weakly separated collection $\C_{\max}$, we can mutate $\C_{\max}$ to change $Hac$ to $Hbd$. So the lemma shows that, if a weakly separated collections looks like it should be mutable in a certain manner, then it can be extended to a maximal weakly separated collection which is mutable in that manner.

\begin{proof}
We will show $Hab \parallel J$, the cases of $Hbc$, $Hcd$ and $Had$ are similar. Assume for the sake of contradiction that $Hab$ and $J$ are not weakly separated. 

From $Hab \not \parallel J$ and $Hac \parallel J$, there is some element of $J \setminus H$ that is in the interval $(b,c]$. Similarly, there is some element of $J \setminus H$ that is in the interval $[d,a)$. Combining these two information with $Hac \parallel J$, gives us that $c \in J$. Similarly, $d \in J$.

Now if $a \in J$, then $Hbd \parallel J$ would imply $Hab \parallel J$. Hence $a \not \in J$, and due to similar reason, $b \not \in J$.

For $Hab \not \parallel J$ to happen despite $Hac \parallel J$ and $Hbd \parallel J$ being true, we need to have some element of $J \setminus H$, say $q$, in the interval $(a,b)$. From $Hac \parallel J$, we get $Hac \setminus J \subset [d,q]$. From $Hbd \parallel J$, we get $Hbd \setminus J \subset [q,c]$. Combining these two facts, we get $H \setminus J \subset [d,q] \cap [q,c]$. Therefore, $H \setminus J$ has to be an empty set, and we get a contradiction since $|Hab| = |J|$.

\end{proof}

\section{Main result}

In this section, we will prove our main result: if $\C_1$ and $\C_2$ are maximal weakly separated collections of some positroid $\M_{\I}$, then $\C_1$ can be mutated to $\C_2$ while preserving the sets they have in common. Throughout this section, we will fix a positroid $\M$ and its Grassmann necklace $\I$. 

Let $\B$ be a weakly separated collection contained in $\W$.  We will say that two maximal weakly separated collections $\C$ and $\C'$ are \newword{$\B$-equivalent within $\M_{\I}$} if there is a chain of maximal weakly separated collections $\C = \C^1 \rightarrow \C^2 \rightarrow \cdots \rightarrow \C^{q-1} \rightarrow \C^{q} = \C'$, such that:
\begin{itemize}
\item $\B \subseteq \C^1,\ldots,\C^q$,
\item $\C^{i+1}$ is obtained from $\C^i$ by one mutation move.
\item All the $\C^i$ obey $\I \subseteq \C^i \subseteq \M_{\I}$.
\end{itemize}

\begin{Theorem} \label{thm:main}
If $\C$ and $\C'$ are maximal weakly separated collections within $\M_{\I}$ containing $\B$, then $\C$ and $\C'$ are $\B$-equivalent within $\M_{\I}$. 
\end{Theorem}

\begin{proof}
Since weakly separated collections in $\M_{\I}$ contain $\I$ by definition, we may assume that $\I \subset \B$.
So the condition $\I \subseteq \C^i$ will follow from $\B \subseteq \C^i$, and we will only need to check that the $\C^i$ are contained in $\M_{\I}$.

Let $\Sigma(\B)$ be the $2$-dimensional CW-complex defined in the previous section associated to $\B$. We fix a map $\pi$ as in the previous section. 
So $\pi(\Sigma(\B))$ is a closed region of $\RR^2$, whose exterior boundary is $\pi(\I)$.
Let $A(\B)$ be the area of the bounded regions of $\RR^2 \setminus \pi(\Sigma(\B))$.
Let $\delta = \min \mathrm{Area}(\mathrm{Hull}(v_a, v_b, v_c))$ where the minimum is over $1 \leq a < b < c \leq n$. 
So $\delta$ is the smallest possible area of triangle appearing in a triangulation of some $\Sigma(\C)$.
Our proof is by induction on $\lceil A(\B) /\delta \rceil$. If $t:= \lceil A(\B) /\delta \rceil=0$ then $A(\B)=0$ and $\Sigma(\B)$ fills the entire interior of $\pi(\I)$, so $\B$ is maximal in $\M_{\I}$ and $\B$ is the only maximal weakly separated collection in $\M_{\I}$ containing $\B$, so the Theorem is vacuously true.

Now, suppose that $A(\B)>0$. So there is some hole within $\pi(\Sigma(\B))$. 
Let $K$ and $L$ be the $k$-element sets labeling two consecutive elements on the boundary of the hole.
Let $\C_1$ and $\C_2$ be two maximal weakly separated collections in $\M_{\I}$ containing $\B$. Then $K$ and $L$ lie in a common face of $\Sigma(\C_r)$ (for $r=1$, $2$.)
Triangulate $\Sigma(\C_r)$ using the edge $(K,L)$. Let $J_r$  be the third vertex of the triangle of $\Sigma(\C_r)$ containing $(K,L)$ and lying on the hole side. Let $T_r$ be the triangle $(J_r, K, L)$. We now divide into 2 cases depending on the colors of the triangles $T_r$.

\textbf{Case 1:}  $T_1$ and $T_2$ are both white or both black. We present the case that the triangles are white; the other case is very similar. 
Set $H = K \cap L$.
Then the $J_r$ are of the forms $H  e_r$ for some $e_1$ and $e_2$. 
From this we can compute that $J_1$ and $J_2$ are weakly separated from each other. Also, by hypothesis, $\B \cup \{ J_1 \}$ and $\B \cup \{ J_2 \}$ are weakly separated. 
So $\B \cup \{ J_1, J_2 \}$ is weakly separated; complete $\B \cup \{ J_1, J_2 \}$ to a maximal weakly separated collection $\C'$ in $\M_{\I}$.

Set $\B_r = \B \cup \{ J_r \}$. Then $\Sigma(\B_r)$ is $\Sigma(\B)$ with an extra triangle added on, so $A(\B_r) \leq A(\B) - \delta$. 
Now, $\C'$ and $\C_r$ contain $\B_r$. So, by induction, $\C_r$ is $\B_r$-equivalent to $\C'$ within $\M_{\I}$. 
Connecting the chains $\C_1 \to \cdots \to \C' \to \cdots \to \C_2$, we see that $\C_1$ and $\C_2$ are $\B$-equivalent within $\M_{\I}$.

\textbf{Case 2:}  $T_1$ is white and $T_2$ is black:
Then we can write $(J_1, J_2, K, L)$ as $(Hac, Hbd, Hab, Had)$. Since $(J_1, K, L)$ and $(J_2, K, L)$ are oriented the same way, the triples $(c,b,d)$ and $(a,d,b)$ are cyclically oriented the same way, which shows that $(a,b,c,d)$ are cyclicly oriented.
By Lemma~\ref{lem:cross},  $Hab$, $Hac$, $Had$, $Hbc$, $Hbd$ and $Hcd$ are weakly separated from $\B$.
Set $\B_1 = \B \cup  \{ Hac, Hab, Had, Hbc, Hcd\}$ and $\B_2 = \B \cup  \{ Hbd, Hab, Had, Hbc, Hcd\}$, so the $\B_r$ are weakly separated. 
Moreover, in $\Sigma(\B_r)$, the new vertices that we have added lie immediately adjacent to the edge $(Hab, Had)$ of $\Sigma(\B)$, inside the hole of $\Sigma(\B)$, and hence lie inside $\pi(\I)$. So, by Lemma~\ref{topsame}, these new vertices lie in $\M_{\I}$, so $\B_1$ and $\B_2$ are weakly separated collections in $\M_{\I}$.

Complete $\B_1$ to a maximal weakly separated collection $\C'_1$ within $\M_{\I}$; define $\C'_2$ to be the mutation of $\C'_1$ where we replace $Hac$ by $Hbd$. 
Then $\C_1 \cap \C'_1 \supseteq \B \cup \{ Hac \}$ and $\C_2 \cap \C'_2 \supseteq \B \cup \{ Hbd \}$. 
The complexes $\Sigma( \B \cup \{ Hac \})$ and $\Sigma(\B \cup \{ Hbd \})$ are $\Sigma(\B)$ with one added triangle.
So, by induction, $\C_1$ and $\C'_1$ are $( \B \cup \{ Hac \})$-equivalent within $\M_{\I}$, and $\C_2$ and $\C'_2$ are $( \B \cup \{ Hbd \})$-equivalent within $\M_{\I}$.
Chaining together the mutations $\C_1 \to \cdots \to \C'_1 \to \C'_2 \to \cdots \to \C_2$, we see that $\C_1$ and $\C_2$ are $\B$-equivalent.

\end{proof}

\section{Implications of the main result}

In this section, we go over the direct implications of Theorem~\ref{thm:main}. Each maximal weakly separated collection corresponds to a reduced plabic graph and the mutation of maximal weakly separated collections corresponds to \newword{square moves} of plabic graphs \cite{OPS}. 

We can define a \newword{plabic complex} of a positroid $\M$. Consider a simplicial complex where the vertices are labeled with Pl\"ucker coordinates and the facets are given by plabic graphs (maximal weakly separated collections) of $\M$. A special case of this complex, when $\M$ is the uniform matroid ${[n] \choose k}$, was studied in \cite{HH}. Hess and Hirsch also conjectured that the complex is a pseudomanifold with boundary.

A simplicial complex is a \newword{pseudomanifold with boundary} if it satisfies the following properties \cite{S}:
\begin{itemize}
\item (pure) The facets have the same dimension.
\item (non-branching) Each codimension 1 face is a face of one or two facets.
\item (strongly connected) Any two facets can be joined by a chain of facets in which each pair of neighboring facets have a common codimension 1 face.
\end{itemize}

Therefore, another way to interpret Theorem~\ref{thm:main} is:
\begin{corollary}
Let $\M$ be a positroid. The plabic complex of $\M$ is a pseudomanifold with a boundary.
\end{corollary}

A similar phenomenon for wiring diagrams and maximal strongly separated collections will be shown in \cite{O2}.



%

\section{Acknowledgments}
The first author would like to thank Sergey Fomin for useful discussions.


\end{document}